\documentclass[english]{article}
\usepackage[utf8,latin9]{inputenc}
\usepackage{arxiv}
\usepackage{hyperref}       
\usepackage{url}            
\usepackage{amsmath,amssymb,amsthm,graphicx}
\usepackage{booktabs}       
\usepackage{amsfonts}       
\usepackage{nicefrac}       
\usepackage{microtype}      
\usepackage{dsfont}
\usepackage{color}
\usepackage[T1]{fontenc}

\usepackage{multirow}
\usepackage{amstext}
\usepackage[authoryear]{natbib}



\usepackage{babel}

\newtheoremstyle{thm}
{9pt}
{9pt}
{\itshape}
{}
{\bfseries}
{.}
{ }
{}
\theoremstyle{thm}
\newtheorem{theorem}{Theorem}[section]
\newtheorem{lemma}[theorem]{Lemma}

\newtheoremstyle{def}
{9pt}
{9pt}
{}
{}
{\bfseries}
{.}
{ }
{}
\theoremstyle{def}

\newtheorem{remark}[theorem]{Remark}
\newtheorem{example}[theorem]{Example}

\newcommand{\R}{\mathbb{R}} 
\newcommand{\N}{\mathbb{N}} 
\newcommand{\E}{\mathbb{E}} 
\newcommand{\PP}{\mathbb{P}} 

    \def\cd{\stackrel{\mathcal{D}}{\longrightarrow}}
    \def\cp{\stackrel{\mathcal{\PP}}{\longrightarrow}}
\renewcommand{\footnoterule}{%
	\kern -3.5pt
	\hrule width \textwidth height 1pt
	\kern 3.5pt
}

\renewcommand{\footnoterule}{%
	\kern -3.5pt
	\hrule width \textwidth height 1pt
	\kern 3.5pt
}

\makeatletter
\def\blfootnote{\xdef\@thefnmark{}\@footnotetext}
\makeatother

\title{The test of exponentiality based on the mean residual life function revisited}

\author{Bruno Ebner\\
 Institute of Stochastics, \\
Karlsruhe Institute of Technology (KIT), \\
Englerstr. 2, 76133 Karlsruhe, \\
Germany\\
\href{mailto:Bruno.Ebner@kit.edu}{Bruno.Ebner@kit.edu}}

\date{\today}

\begin{document}

\maketitle


\begin{abstract}
We revisit the family of goodness-of-fit tests for exponentiality  based on the mean residual life time proposed by \cite{BH:2008}. We motivate the test statistic by a characterisation of \citet{S:1070} and provide an alternative representation, which leads to simple and short proofs for the known theory and an easy to access covariance structure of the limiting Gaussian process under the null hypothesis. Explicit formulas for the eigenvalues and eigenfunctions of the operator associated with the limit covariance are given using results on weighted Brownian bridges. In addition we derive further asymptotic theory under fixed alternatives as well as approximate Bahadur efficiencies, which provide an insight into the choice of the tuning parameter with regard to the power performance of the tests.    \end{abstract}

\section{Introduction}\label{sec:Intro}
We revisit the family of tests for exponentiality based on the residual life time as proposed by \cite{BH:2008}, and we provide further theoretical insight into the asymptotic behaviour of the test under alternatives. The problem of interest is testing the assumption that data are distributed according to the exponential distribution with unknown scale parameter. To be precise, let $X$ be a positive random variable and we write shorthand $X\sim\mbox{Exp}(\lambda)$, $\lambda > 0$, if $X$ follows an exponential distribution with scale parameter $\lambda$, hence the density is given by
\begin{equation}\label{eq:density}
f(x,\lambda) = \lambda\exp(-\lambda x),\quad x>0.
\end{equation}
Note that $X \sim$ $\mbox{Exp}(\lambda)$ if, and only if, $X/\lambda \sim$ $\mbox{Exp}(1)$, which shows that the exponential distribution belongs to the scale family of distributions, for a detailed discussion see \cite{JKB:1994}, Chapter 19. In the following we denote the family of exponential distributions by $\mathcal{E}:=\{\mbox{Exp}(\lambda):\,\lambda>0\}$. Let $X, X_1, X_2, \dotso$ be positive independent and identically distributed (iid.) random variables with distribution $\mathbb{P}^X$ defined on an underlying probability space $(\Omega,\mathcal{A},\mathbb{P})$. We test the composite hypothesis
\begin{equation}\label{eq:H0}
H_0:\;\mathbb{P}^X\in\mathcal{E}
\end{equation}
against general alternatives, based on the sample $X_1,\ldots,X_n$.

This testing problem has been studied extensively in the literature, and it is of ongoing interest, see \cite{CMO:2021,CMNO:2020,JMO:2020,MO:2016b,PWG:2022,VG:2020} for some recent related publications, \cite{A:1990,HM:2005,S:1984} for surveys on the topic, and \cite{ASSV:2017,OMO:2022} for reviews of testing procedures as well as of extensive competitive Monte Carlo power studies. The focus of this article is on the procedure proposed in \cite{BH:2008}. That paper studies a family of tests of exponentiality based on the mean residual life function with test statistic
\begin{equation*}
    G_{n,a}=n\int_{0}^\infty \left[\frac1n\sum_{j=1}^n\min(Y_{j},z)-\frac1n\sum_{j=1}^n\mathbf{1}\{Y_{j}\le z\}\right]^2\exp(-az)\,\mbox{d}z,\quad a>0.
\end{equation*}
Here, $Y_j=X_j/\overline{X}_n$, $j=1,\ldots,n$, is the scaled data, where $\overline{X}_n=\frac1n\sum_{j=1}^n X_j$. This family of tests is an extension of the test in \cite{BH:2000} insofar as the exponential weight function in the original test statistic is replaced with a more flexible weight function $w(z)=e^{-az}$, $z>0$, which depends on some so-called tuning parameter $a>-1$. In \cite{BH:2008} the authors provide the limiting null distribution of the test statistics $G_{n,a}$ and a simulation study including simulated critical values, as well as a power study for different values of the tuning parameter $a$. Further insight on the behaviour of the tests under alternatives is hitherto missing. We fill this gap in the literature.

We start our investigation by providing an alternative representation of $G_{n,a}$ motivated by a characterisation of the exponential law due to \citet{S:1070}, which is obviously related to the mean residual life function. \citet{S:1070}
states that for a random variable $X$ with $\mathbb{P}(X>y)>0$ for $y>0$, we have $X\sim\mbox{Exp}(\lambda)$ if and only if
\begin{equation*}
\mathbb{E}[X|X>y]=y+\frac1\lambda,\quad \mbox{for all } y>0, \lambda>0.
\end{equation*}
Direct calculations show that under the same condition this characterisation can be restated in the following way: $X\sim\mbox{Exp}(\lambda)$ if and only if
\begin{equation}\label{eq:Char}
\mathbb{E}\left[\left(X-y-\frac1\lambda\right)\mathbf{1}\{X>y\}\right]=0,\quad \mbox{for all } y>0, \lambda>0.
\end{equation}
In view of the scale invariance of the family $\mathcal{E}$, we consider the scaled data $Y_j=X_j/\overline{X}_n$, $j=1,\ldots,n$, and propose the weighted $L^2$-type test statistic (fixing w.l.o.g. $\lambda=1$ in the characterisation)
\begin{equation}\label{eq:Tna}
T_{n,a}=T_{n,a}(Y_1,\ldots,Y_n)=n\int_0^\infty\left|\frac1n\sum_{j=1}^n(Y_j-y-1)\mathbf{1}\{Y_j>y\}\right|^2\exp(-a y)\,\mbox{d}y,
\end{equation}
which depends on the tuning parameter $a\ge0$. Note that $T_{n,a}$ is scale invariant, i.e. invariant w.r.t. transformations of the form $x\mapsto b x$, $b>0$, since it only depends on the scaled data $Y_1,\ldots,Y_n$. Scale invariance is indeed a desirable property, since the family $\mathcal{E}$ is closed under such transformations.
Putting $x\wedge y=\min(x,y)$ for real numbers $x,y$ and using $n^{-1}\sum_{j=1}^nY_j=1$, some algebra shows 
\begin{equation}\label{eq:expl}
T_{n,a}=\frac1n\sum_{j,k=1}^n\frac1a(|Y_j-Y_k|-1)e^{-a Y_j\wedge Y_k}+\frac1{a^2}(Y_j+Y_k-2(Y_j\wedge Y_k+1))e^{-a Y_j\wedge Y_k}-\frac2{a^3}\left(e^{-a Y_j\wedge Y_k}-1\right),
\end{equation}
for each $a>0$, as well as 
\begin{equation*}
T_{n,0}=\frac1n\sum_{j,k=1}^n\frac13(Y_j\wedge Y_k)^3-\frac{Y_j+Y_k-2}{2}(Y_j\wedge Y_k)^2+(Y_j-1)(Y_k-1)(Y_j\wedge Y_k).
\end{equation*}
An efficient implementation of \eqref{eq:expl} in the statistical computing language \texttt{R}, see \cite{R:2021}, can be found in Appendix \ref{app:SC}.

\begin{remark}\label{rem:Intro} 
\begin{enumerate}
    \item Some algebra shows that $T_{n,a}\equiv G_{n,a}$ holds for all $a>-1$, hence both test statistics are identical.
    
    \item Representation \eqref{eq:Tna} leads to a new compact and formula of the covariance kernel of the limiting Gaussian process, which admits direct conclusions for the limiting distribution of the test statistic.
    
    \item Obviously the restriction $a\ge0$ can be generalised to negative values. In this article we mostly focus on the positive half axis, since the power results in the simulation study of \cite{BH:2008} suggest that negative tuning parameters lead to low power of the tests. The authors suggest to use the tuning parameter $1\le a\le 2$ as a generally good choice.
    
    \item Note that the test statistic $T_{n,0}$ is related to the test of exponentiality based on the integrated distribution function proposed in \cite{K:2001}, Section 2. The main difference in the definitions of the tests lies in the fact, that $\mathbb{P}(X>y)$ in \eqref{eq:Char} is estimated by and not replaced with the theoretical known value under $H_0$. This difference has an impact on the limiting distribution and the power of the tests as is shown in the sequel.
\end{enumerate} 
\end{remark}

The rest of the paper is organised as follows. In Section \ref{sec:Limit dist} we provide a simple and direct proof of the asymptotic limiting distribution of $T_{n,a}$ under the null hypothesis and provide the first four cumulants of the limiting null distribution as well as explicit formulas of the eigenvalues and eigenfunctions associated with the integral operator induced by the covariance structure of the limiting Gaussian process. These findings corroborate the results of \cite{BH:2008} in a very short and direct way.
In Section \ref{sec:Limit_dist_fixed_alt} we provide new limit results under fixed alternatives as well as a proof of the consistency of the testing procedures. Local approximate Bahadur efficiencies are deduced and interpreted in connection to empirical power results in Section \ref{sec:ABE}. We conclude the article by pointing out open problems for further research in Section \ref{sec:conc} and by an Appendix containing explicit formulas and some source code.

\section{Limiting distribution under the null hypothesis}\label{sec:Limit dist}
In this section we assume that $X_1,X_2,\ldots$ is an iid. sequence of random variables with $X_1\sim\mbox{Exp}(1)$. Due to the $L^2$-structure of the test statistics, a convenient setting --is in dependence of the constant $a\ge0$-- the separable Hilbert space $\mathbb{H}_a=L^2([0,\infty),\mathcal{B},\exp(-at){\rm d}t)$ of (equivalence classes of) measurable functions $f:[0,\infty) \rightarrow \R$
satisfying $\int_0^\infty |f(t)|^2 \exp(-at)\, {\rm d}t < \infty$. Here, $\mathcal{B}$ stands for the Borel sigma-field on $[0,\infty)$. The scalar product and the norm in $\mathbb{H}_a$ will be denoted by
\begin{equation*}
\langle f,g \rangle_{\mathbb{H}_a} = \int_0^\infty f(t)g(t)\,\exp(-at)\, {\rm d}t, \quad \|f\|_{\mathbb{H}_a} = \langle f,f \rangle_{\mathbb{H}_a}^{1/2}, \quad f,g \in \mathbb{H}_a,
\end{equation*}
respectively. 

A change of variable shows that
\begin{equation*}
T_{n,a}=\overline{X}_n^{-3}\int_0^\infty Z_n^2(y)\exp(-ay/\overline{X}_n)\,\mbox{d}y,
\end{equation*}
where 
\begin{equation*}
Z_n(y)=\frac1{\sqrt{n}}\sum_{j=1}^n(X_j-y-\overline{X}_n)\mathbf{1}\{X_j>y\},\quad y>0.
\end{equation*}
Define the auxiliary process
\begin{equation*}
Z_n^\ast(y)=\frac1{\sqrt{n}}\sum_{j=1}^n(X_j-y-1)\mathbf{1}\{X_j>y\}-(X_j-1)\mathbb{P}(X_1>y),\quad y>0.
\end{equation*}
Notice that $Z_n^*$ is a sum of centred iid. random variables, and we have $\mathbb{E}[Z_1^*(s)Z_1^*(t)]=K_0(s,t)$, $s,t>0$, where $K_0$ defined below. By the central limit theorem in Hilbert spaces there exists a centred Gaussian process $Z$ in $\mathbb{H}_a$ having covariance kernel 
\begin{equation}\label{eq:K0}
    K_0(s,t)=\exp(-s\vee t)-\exp(-(s+t)),\quad s,t>0,
\end{equation}
such that $Z_n^*\cd Z$ in $\mathbb{H}_a$ as $n\to\infty$. Here and in what follows, we write $s\vee t=\max(s,t)$, and $\cd$ denotes convergence in distribution. Likewise $\cp$ will denote convergence in probability. Note that $K_0$ is an alternative representation of the limiting covariance kernel $\rho$ in (7) of \cite{BH:2008}.  

\begin{theorem}\label{thm:AsyDist}
Under the standing assumptions, there exists a centred Gaussian process $Z$ in $\mathbb{H}_a$ with covariance kernel $K_0$ defined in \eqref{eq:K0}, such that 
\begin{equation*}
T_{n,a}\cd \|Z\|_{\mathbb{H}_a}^2,\quad\mbox{as}\,n\rightarrow\infty.
\end{equation*}
\end{theorem}
\begin{proof}
We first consider the case $a=0$. Note that 
\begin{eqnarray*}
    \|Z_n-Z_n^*\|_{\mathbb{H}_0}^2&=&\left(\frac{1}{n}\sum_{j=1}^n(X_j-1)\right)^2\int_0^\infty\left|\frac{1}{\sqrt{n}}\sum_{k=1}^n\left(\mathbf{1}\{X_k>y\}-\mathbb{P}(X_1>y)\right)\right|^2\mbox{d}y.
\end{eqnarray*}
By the strong law of large numbers $n^{-1}\sum_{j=1}^n(X_j-1)$ converges to 0 a.s., and the central limit theorem in $\mathbb{H}_0$ implies that the second term is a tight sequence. Hence, $\|Z_n-Z_n^*\|_{\mathbb{H}_0}^2=o_{\mathbb{P}}(1)$, and $T_{n,0}\cd \|Z\|_{\mathbb{H}_0}^2$, as $n\rightarrow\infty$ follows from $Z_n^*\cd Z$ in $\mathbb{H}_0$, Slutski's lemma and the continuous mapping theorem.

For the case $a>0$ define
\begin{equation*}
\widetilde{T}_{n,a}=\overline{X}_n^{-3}\int_0^\infty Z_n^2(y)\exp(-ay)\,\mbox{d}y.
\end{equation*}
A first-order Taylor expansion yields
\begin{equation}\label{eq:Texp1}
    \exp(-ay/\overline{X}_n)=\exp(-ay)+\frac{ay}{\Delta_n^2}\exp\left(-\frac{ay}{\Delta_n}\right)\left(\overline{X}_n-1\right),
\end{equation}
where $\Delta_n\in(\min(\overline{X}_n,1),\max(\overline{X}_n,1))$.
From the Cauchy-Schwarz inequality and \eqref{eq:Texp1} we obtain
\begin{eqnarray*}
|T_{n,a}-\widetilde{T}_{n,a}|&=&\overline{X}_n^{-3}\int_0^\infty Z_{n}^2(y)|\exp(-ay/\overline{X}_n)-\exp(-ay)|\mbox{d}y\\
&\le&\overline{X}_n^{-3}\int_0^\infty Z_{n}^2(y)\mbox{d}y \int_0^\infty|\exp(-ay/\overline{X}_n)-\exp(-ay)|\mbox{d}y\\
&=&\overline{X}_n^{-3}T_{n,0}\int_0^\infty\frac{ay}{\Delta_n^2}\exp\left(-\frac{ay}{\Delta_n}\right)\left|\overline{X}_n-1\right|\mbox{d}y\\
&=&\frac1a \overline{X}_n^{-3}T_{n,0}\left|\overline{X}_n-1\right|.
\end{eqnarray*}
The strong law of large numbers implies $\overline{X}_n\to 1$ a.s., and since $T_{n,0}$ is a tight sequence (see case $a=0$), \linebreak $|T_{n,a}-\widetilde{T}_{n,a}|=o_{\mathbb{P}}(1)$ follows. By the central limit theorem in $\mathbb{H}_a$ there exists a centred Gaussian process $\widetilde{Z}$ with covariance kernel $K_0$ such that $Z_n^*\cd \widetilde{Z}$ as $n\to\infty$ in $\mathbb{H}_a$ and the same reasoning as for the case $a=0$ yields $\|Z_n-Z_n^*\|_{\mathbb{H}_a}^2=o_{\mathbb{P}}(1)$. Invoking Slutzki's lemma and the continuous mapping theorem, we obtain $\widetilde{T}_{n,a}\cd \|\widetilde{Z}\|_{\mathbb{H}_a}^2$ from which the claim follows.

\end{proof}

\begin{remark}\label{rem:cum}
 By direct evaluation of integrals the first two cumulants of the distribution of $\|Z\|_{\mathbb{H}_a}^2$ are 
\begin{equation*}
\kappa_{1,a}=\E \|Z\|_{\mathbb{H}_a}^2=\int_0^\infty K_0(t,t)\exp(-at)\,\mbox{d}t=\frac1{(a+1)(a+2)}
\end{equation*}
and
\begin{equation*}
\kappa_{2,a}=\mathbb{V}(\|Z\|_{\mathbb{H}_a}^2)=2\int_0^\infty\int_0^\infty K_0^2(s,t)\exp(-a(s+t))\,\mbox{d}s\mbox{d}t=\frac{2}{(a^2+3a+2)(2a+3)(a+2)},
\end{equation*}
where $\mathbb{V}(\cdot)$ denotes the variance. Following the methodology in \cite{H:1990,S:1976}, we calculate the third and fourth cumulants according to 
\begin{equation*}
\kappa_{j,a}=2^{j-1}(j-1)!\int_{0}^\infty K_j(t,t) \exp(-at)\,\mbox{d}t,
\end{equation*}
where $K_j(s,t)$, the $j^{\text{th}}$ iterate of $K_0(s,t)$, is given by
\begin{eqnarray*}
K_{j}(s,t)&=&\int_0^\infty K_{j-1}(s,u)K_0(u,t)\exp(-au)\,\mbox{d}u,\quad j\ge2,\\
K_1(s,t)&=&K_0(s,t).
\end{eqnarray*}
Direct calculation shows that 
\begin{equation*}
    \kappa_{3,a}=\frac{16}{(1+a)(2a+3)(a+2)^3(3a+4)}
\end{equation*}
and 
\begin{equation*}
\kappa_{4,a}=\frac{48(11a+16)}{(3a+4)(2a+3)^2(a+2)^4(1+a)(4a+5)}.
\end{equation*}
These formulas are very useful in order to fit a Pearson system of distributions to approximate the critical values of the test statistic $T_{n,a}$. The approximated quantiles of the limiting distribution are found in Table \ref{tab:cv} and source code is provided in Appendix \ref{app:SC}. A comparison with Table 1 in \cite{BH:2008} shows that the approximation of the critical values is a reasonable fit to the empirical critical values for $a\ge0$.
\end{remark}
\begin{table}[t]
\centering
\begin{tabular}{lrrrrrrrrrr}
$q/a$ & -0.99 & -0.9 & -0.5 & 0 & 0.5 & 1 & 1.5 & 2 & 5 & 10 \\ 
  \hline
0.9 & 117.263 & 13.919 & 2.522 & 1.009 & 0.553 & 0.351 & 0.243 & 0.178 & 0.052 & 0.017 \\ 
  0.95 & 123.673 & 16.162 & 3.189 & 1.309 & 0.725 & 0.463 & 0.322 & 0.237 & 0.069 & 0.022 \\ 
  0.99 & 136.942 & 21.396 & 4.813 & 2.045 & 1.149 & 0.739 & 0.516 & 0.381 & 0.112 & 0.036 
\end{tabular}
\caption{Approximated $q$ - quantiles of the limiting distribution in Theorem \ref{thm:AsyDist} by a Pearson system of distributions using the formulas of the cumulants in Remark \ref{rem:cum}.}\label{tab:cv}
\end{table}
From the theory of Gaussian processes it is well known (see \cite{SW:1986}, p. 206) that for each $a>0$ an orthogonal decomposition of the process $Z$ yields $\|Z\|^2_{\mathbb{H}_a}=\sum_{j=1}^\infty \lambda_j(a) N_j^2,$ where $N_1,N_2,\ldots$ are iid. standard normal, and  $\lambda_1(a),\lambda_2(a),\ldots$ is the decreasing sequence of positive eigenvalues of the integral operator $\mathcal{K}_a:\mathbb{H}_a\rightarrow\mathbb{H}_a,\, f\mapsto\mathcal{K}_af(\cdot)=\int_0^\infty K_0(\cdot,t)f(t)\exp(-at)\mbox{d}t.$
To calculate the eigenvalues $\lambda_j(a)$, $j=1,2,\ldots$, of $\mathcal{K}_a$, one has to solve the homogeneous Fredholm integral equation of the second kind
\begin{equation*}\label{eq:inteq}
\int_0^\infty K_0(x,t)f(t)\exp(-at)\mbox{d}t=\lambda f(x),\quad x>0,
\end{equation*}
see, for example, \cite{KS:1947}. Usually these problems are very hard to solve explicitly and numerical or Monte Carlo simulation techniques are used to obtain approximations of $\lambda_j(a)$, see Section 5 in \cite{EH:2021} for a stochastic approximation method or the method presented in \cite{BMNO:2020}.

In the following we give explicit formulas for general tuning parameters $a$. Firstly, note that $K_0$ admits  the representation
\begin{eqnarray}
K_0(s,t)&=& \min(\exp(-s),\exp(-t))-\exp(-s)\exp(-t),\quad s,t>0.\label{eq:altKern1}
\end{eqnarray}

Let $a=1$. Using the alternative representation in \eqref{eq:altKern1} together with $x=\exp(-s)$ and a substitution ($y=\exp(-t)$), we have
\begin{equation*}
\int_0^\infty K_0(x,t)f(t)\exp(-t)\mbox{d}t=\int_0^1(\min(x,y)-xy)f(-\log(y))\mbox{d}y.
\end{equation*}
Since $\mathbb{K}(x,y)=\min(x,y)-xy$, $x,y\in[0,1]$, is the covariance kernel of the Brownian bridge $B(t)$, say,  the eigenvalue problem is solved in this case, see e.g. \cite{AD:1952}. We conclude that the eigenvalues and corresponding eigenfunctions are
\begin{equation*}
    \lambda_k(1)=\frac{1}{(k\pi)^2} \quad\mbox{and}\quad f_{k,1}(t)=\sin\left(k\pi\exp(-t)\right),\quad t>0.
\end{equation*}

By analogy the general case $a>0$ leads to 
\begin{equation*}
\int_0^\infty K_0(x,t)f(t)\exp(-at)\mbox{d}t=\int_0^1(\min(x,y)-xy)f(-\log(y))y^{a-1}\mbox{d}y,
\end{equation*}
which is connected to the eigenvalue problem of the weighted Brownian bridge $t^{a-1}B(t)$. This problem is solved in Theorem 1.4 in \cite{DM:2003}, which states the eigenvalues and eigenfunctions explicitly. Let $\nu=(a+1)^{-1}$ (note that there is a typographical error in \cite{DM:2003}) and denote by $J_\nu$ the Bessel functions of the first kind and by $0<z_{\nu,1}<z_{\nu,2}<\ldots$ the ascending sequence of zeros of $J_{\nu}$, for details on the zeros of Bessel functions see \cite{W:1995}, Chapter XV. We have for $k=1,2,\ldots$
\begin{equation}\label{eq:eig}
    \lambda_k(a)=\left(\frac{2\nu}{z_{\nu,k}}\right)^2\quad \mbox{and}\quad f_{k,a}(t)=\frac{J_\nu\left(z_{\nu,k}\exp(-\frac{t}{2\nu})\right)}{\sqrt{\nu}J_{\nu-1}(z_{\nu,k})}\exp\left(-\left(\frac1{\nu}-1\right)\frac{t}2\right),\quad t>0.
\end{equation}
By using the identities in Remark \ref{rem:cum} (compare to Corollary 1.3 in \citet{DM:2003}) we see that $\sum_{k=1}^\infty\lambda^j_k(a)=\kappa_{j,a}/(2^{j-1}(j-1)!)$, $j=1,2,3,4$. The largest twenty eigenvalues are given in Table \ref{tab:eig}, and for the sake of completeness, we approximated the scaled cumulants by the sum of the respective powers of the first 100 eigenvalues. We see that for the mean, there is still some difference in the first few digits, which is explained by the low speed of convergence to 0 of the eigenvalues. Interestingly, the eigenvalues in Table \ref{tab:eig} in each row are strictly decreasing, which is explained by the fact that for fixed $k$ the function $z_{\nu,k}/\nu$ is increasing for $\nu\downarrow0$, see \cite{Oetal:2010}, p. 236.

\begin{remark}
The stated formulas of eigenvalues and eigenfunctions were also derived in \cite{BH:2000} for the case $a=1$ relating the statistics to the classical Cram\'{e}r-von Mises test and for the general case $a>-1$ by solving a related differential equation in \cite{BH:2008}. From these results we see that \eqref{eq:eig} also holds in these cases.
\end{remark}

\begin{table}[t]
\centering
\begin{tabular}{lrrrrrr}
 $k/a$& 0 & 1 & 2 & 3 & 4 & 5 \\ 
  \hline
1 & 0.2724430 & 0.10132118 & 0.05275301 & 0.03232757 & 0.02183334 & 0.015732912 \\ 
  2 & 0.0812703 & 0.02533030 & 0.01221201 & 0.00716691 & 0.00470790 & 0.003327524 \\ 
  3 & 0.0386475 & 0.01125791 & 0.00528483 & 0.00305755 & 0.00199068 & 0.001398474 \\ 
  4 & 0.0225326 & 0.00633257 & 0.00293284 & 0.00168481 & 0.00109214 & 0.000764966 \\ 
  5 & 0.0147448 & 0.00405285 & 0.00186176 & 0.00106499 & 0.00068857 & 0.000481445 \\ 
  6 & 0.0103955 & 0.00281448 & 0.00128585 & 0.00073348 & 0.00047341 & 0.000330626 \\ 
  7 & 0.0077217 & 0.00206778 & 0.00094101 & 0.00053570 & 0.00034534 & 0.000240981 \\ 
  8 & 0.0059612 & 0.00158314 & 0.00071835 & 0.00040833 & 0.00026299 & 0.000183403 \\ 
  9 & 0.0047409 & 0.00125088 & 0.00056629 & 0.00032151 & 0.00020693 & 0.000144239 \\ 
  10 & 0.0038603 & 0.00101321 & 0.00045785 & 0.00025971 & 0.00016705 & 0.000116401 \\ 
  11 & 0.0032042 & 0.00083737 & 0.00037782 & 0.00021415 & 0.00013768 & 0.000095907 \\ 
  12 & 0.0027021 & 0.00070362 & 0.00031708 & 0.00017960 & 0.00011543 & 0.000080385 \\ 
  13 & 0.0023095 & 0.00059953 & 0.00026989 & 0.00015279 & 0.00009817 & 0.000068347 \\ 
  14 & 0.0019966 & 0.00051694 & 0.00023250 & 0.00013156 & 0.00008450 & 0.000058824 \\ 
  15 & 0.0017433 & 0.00045032 & 0.00020237 & 0.00011447 & 0.00007351 & 0.000051161 \\ 
  16 & 0.0015352 & 0.00039579 & 0.00017774 & 0.00010050 & 0.00006452 & 0.000044903 \\ 
  17 & 0.0013624 & 0.00035059 & 0.00015735 & 0.00008895 & 0.00005709 & 0.000039727 \\ 
  18 & 0.0012171 & 0.00031272 & 0.00014028 & 0.00007927 & 0.00005088 & 0.000035397 \\ 
  19 & 0.0010939 & 0.00028067 & 0.00012584 & 0.00007110 & 0.00004562 & 0.000031738 \\ 
  20 & 0.0009885 & 0.00025330 & 0.00011352 & 0.00006412 & 0.00004114 & 0.000028618 \\ \hline
  $\sum_{j=1}^{100}\lambda_j(a)$ & 0.4959773 & 0.16565850 & 0.08288489 & 0.04974765 & 0.03317179 & 0.023697320 \\ 
  $\kappa_{1,a}$ & 0.5000000 & 0.16666667 & 0.08333333 & 0.05000000 & 0.03333333 & 0.023809524 \\ 
  $\sum_{j=1}^{100}\lambda_j^2(a)$ & 0.0833333 & 0.01111111 & 0.00297619 & 0.00111111 & 0.00050505 & 0.000261643 \\ 
  $\kappa_{2,a}/2$ & 0.0833333 & 0.01111111 & 0.00297619 & 0.00111111 & 0.00050505 & 0.000261643 \\ 
  $\sum_{j=1}^{100}\lambda_j^3(a)$ & 0.0208333 & 0.00105820 & 0.00014881 & 0.00003419 & 0.00001052 & 0.000003934 \\ 
  $\kappa_{3,a}/8$ & 0.0208333 & 0.00105820 & 0.00014881 & 0.00003419 & 0.00001052 & 0.000003934 \\ 
  $\sum_{j=1}^{100}\lambda_j^4(a)$ & 0.0055556 & 0.00010582 & 0.00000777 & 0.00000109 & 0.00000023 & 0.000000061 \\ 
  $\kappa_{4,a}/48$ & 0.0055556 & 0.00010582 & 0.00000777 & 0.00000109 & 0.00000023 & 0.000000061 \\ 
\end{tabular}
\caption{First twenty eigenvalues $\lambda_k(a)$ for different tuning parameters $a$ and sums of powers over the first 100 eigenvalues and the corresponding cumulants using the formulas in Remark \ref{rem:cum}.}\label{tab:eig}
\end{table}

\section{Limiting distribution under fixed alternatives and consistency}\label{sec:Limit_dist_fixed_alt}
In this section we assume that $X_1,X_2,\ldots$ is a sequence of iid. random variables with cumulative distribution function $F$, $\mathbb{E}[X_1]=1$ and $\mathbb{E}[|X_1|^3]<\infty$. The moment condition is motivated by the scale invariance of $T_{n,a}$ and hence there is no loss of generality compared to fixed alternatives with existing third moment.

\begin{theorem}\label{thm:cons}
We have 
\begin{equation*}
\frac{T_{n,a}}{n}\cp \int_0^\infty z^2(y)\exp(-ay)\,{\rm d}y=\Delta_a,\quad a\ge0,\quad \mbox{as}\; n\to\infty,
\end{equation*}
where $z(y)=\mathbb{E}\left[(X_1-y-1)\mathbf{1}\{X_1>y\}\right]$, $y>0$.
\end{theorem}
\begin{proof}
Note that by the same arguments as in the proof of Theorem \ref{thm:AsyDist} we have $n^{-1}\|Z_n-Z_n^*\|_{\mathbb{H}_a}^2\to0$ a.s. Since by the strong law of large numbers in Hilbert spaces $n^{-1/2}Z_n^*\to z$ a.s. in $\mathbb{H}_a$ the claim follows by the continuous mapping theorem and Slutski's lemma.
\end{proof}

Under the assumption $X_1\sim\mbox{Exp}(1)$, the characterisation in \eqref{eq:Char} gives $z\equiv0$ and hence $\Delta_a=0$. Since $\Delta_a$ equals 0 if and only if $X_1$ follows an exponential distribution, we conclude that the tests $T_{n,a}$ are consistent against each alternative distribution with existing first moment.

\begin{example}\label{ex:alt}
In this example we give explicit formulas for $\Delta_a$ for different $\Gamma(\beta,\beta)$ distributions. Note that for an suitable alternative we need $\mathbb{E}[X]=1$ to be satisfied. Direct evaluations show that
\begin{enumerate}
\item if $X\sim\Gamma(2,2)$, then  $\Delta_a=2/(a+4)^3$,
\item if $X\sim\Gamma(3,3)$, then $\Delta_a=2(a^2+30a+252)/(a+6)^5$, and
\item if $X\sim\Gamma(4,4)$, then $\Delta_a=2\left({a}^{4}+56\,{a}^{3}+1344\,{a}^{2}+16640\,a+94720\right)/ \left( a
+8 \right) ^{7}$.
\end{enumerate}
\end{example}
\begin{remark}
Note that higher values of $\Delta_a$ in dependence of the tuning parameter $a$ do not imply greater power of the test against this alternative. To visualise this behaviour, we performed a Monte Carlo (MC) simulation study  with different $\Gamma(\beta,\beta)$ distributions and plot the empirical rejection rate of the tests in Figure \ref{fig:PowG}. Note that under $\Gamma(1,1)=\mbox{Exp}(1)$, hence under the null hypothesis, the type I error is well calibrated below the significance level of 0.05 (here visualised by a solid line).
\end{remark}
\begin{figure}[t]
    \centering
    \includegraphics[scale=0.4]{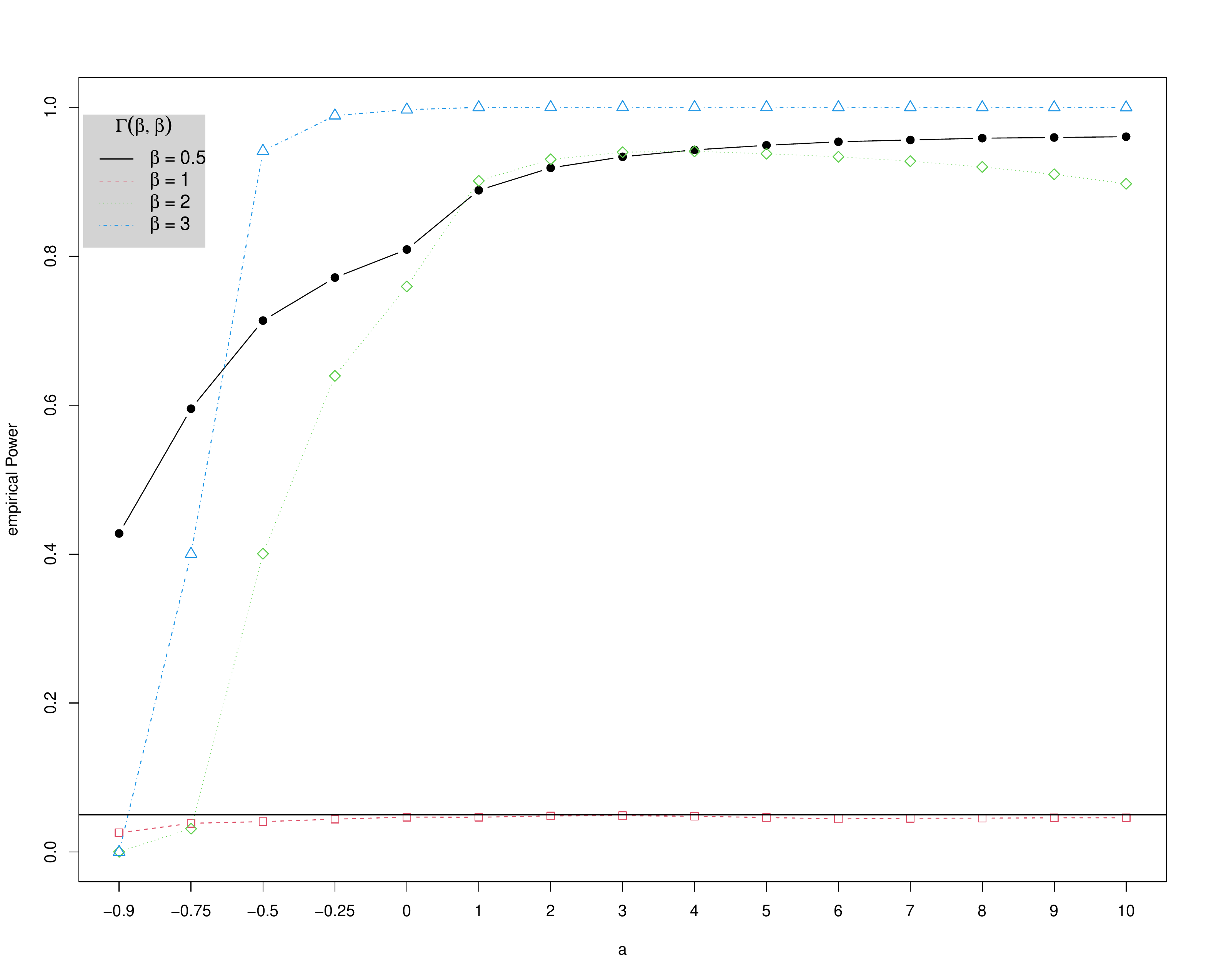}
    \caption{Empirical power of the test $T_{n,a}$ for different tuning parameters $a$. Each dot represents the empirical rejection rate of the test at a significance level of 0.05 under the stated $\Gamma(\beta,\beta)$ distribution (MC simulation with 10000 replications) .}
    \label{fig:PowG}
\end{figure}
In the following we derive the limiting distribution of the test statistic under fixed alternatives under the stated assumptions in the beginning of the section.
Define
\begin{equation*}
W_n(y)=Z_n(y)-\sqrt{n}z(y),\quad W_n^*(y)=Z_n^*(y)-\sqrt{n}z(y),\quad y>0.
\end{equation*}
It is easy to see that $\|W_n-W_n^*\|^2_{\mathbb{H}_a}=o_\mathbb{P}(1)$. 

\begin{lemma}\label{lem:conW}
Under the standing assumptions there exists a centred Gaussian process $W$ in $\mathbb{H}_a$ with covariance kernel
\begin{eqnarray*}
    K(s,t)&=&\Psi_2(s\lor t)+st\mathbb{P}(X>s\lor t) - (s+t)\Psi_1(s\lor t)\\
    &&-\mathbb{P}(X>s)\left(\Psi_2(t)-t\Psi_1(t)\right)-\mathbb{P}(X>t)\left(\Psi_2(s)-s\Psi_1(s)\right)\\
    &&+\mathbb{V}(X)\mathbb{P}(X>s)\mathbb{P}(X>t)-z(s)z(t),\quad s,t>0,
\end{eqnarray*}
where $s\lor t=\max(s,t)$, and $\Psi_\ell(s)=\mathbb{E}[(X-1)^\ell\mathbf{1}\{X>s\}]$, $s>0$, for $\ell=1,2$. In $\mathbb{H}_a$ we have
\begin{equation*}
W_n\cd W\quad \mbox{as}\,n\rightarrow\infty.
\end{equation*}
\end{lemma}
\begin{proof}
Since $\|W_n-W_n^*\|^2_{\mathbb{H}_a}=o_\mathbb{P}(1)$ the limit process is determined by the limit behaviour of $W_n^*$. We have
\begin{equation*}
    W_n^*(y)=\frac1{\sqrt{n}}\sum_{j=1}^n\big[(X_j-y-1)\mathbf{1}\{X_j>y\}-(X_j-1)\mathbb{P}(X_j>y)-z(y)\big],\quad y\ge0,
\end{equation*}
and $\mathbb{E}[W_n^*(y)]=0$, $y>0$, since $\mathbb{E}[X_1]=1$. Hence $W_n^*$ is a sum of iid. elements in $\mathbb{H}_a$ and thus converges as $n\to\infty$ to a centred Gaussian process $W$ with covariance kernel $K(s,t)=\mathbb{E}[W_1^*(s)W_1^*(t)]$, $s,t\ge0$, in $\mathbb{H}_a$ by the central limit theorem in Hilbert spaces. The formula for $K$ is obtained by tedious but straightforward calculations.
\end{proof}

By the results of Lemma \ref{lem:conW} and Theorem 1 in \cite{BEH:2017}, we thus have
\begin{equation}\label{eq:altasyvert}
    \sqrt{n}\left(\frac{T_{n,a}}{n}-\Delta_a\right)\cd \mbox{N}(0,\sigma_a^2),
\end{equation}
where 
\begin{equation*}
\sigma_a^2=4\int_0^\infty\int_0^\infty K(x,y)z(x)z(y)\exp(-a(x+y))\,\mbox{d}x\mbox{d}y.
\end{equation*}
For some families of distributions, $\Delta_a$ and $\sigma_a^2$ can be calculated explicitly for fixed tuning parameters see Example \ref{ex:alt}.

In general, however, we have to find a consistent estimator $\widehat{\sigma}_{n,a}^2$ of $\sigma_a^2$. In this spirit, we replace the probabilities by relative frequencies, the variance by the empirical variance $s_y^2=\frac1n\sum_{j=1}^n(Y_j-1)^2$, and the expectation by empirical counterparts
\begin{equation*}
\Psi_{n,\ell}(s)=\frac1n\sum_{j=1}^n(Y_j-1)^\ell\mathbf{1}\{Y_j>s\},\quad \ell=1,2,\quad \mbox{and}\quad z_{n}(s)=\frac1n\sum_{j=1}^n(Y_j-s-1)\mathbf{1}\{Y_j>s\},\quad s>0,
\end{equation*}
based on the scaled random variables $Y_1,\ldots,Y_n$. Denoting by $K_n$ the resulting estimator of $K$ obtained from plugging-in the empirical counterparts into the formula of $K$, the estimator  $\widehat{\sigma}_{n,a}^2$ of $\sigma_a^2$ is
\begin{equation*}
\widehat{\sigma}_{n,a}^2=4\int_0^\infty\int_0^\infty K_n(x,y)z_n(x)z_n(y)\exp(-a(x+y))\,\mbox{d}x\mbox{d}y.
\end{equation*}
Writing $\widehat{\tau}_n^2=n^{-1}\sum_{j=1}^n(Y_j-1)^2$, and
\begin{eqnarray*}
S_{1n}&=&\int_0^\infty\int_0^\infty \Psi_{n,2}(s\lor t)z_n(s)z_n(t)\exp(-a(s+t))\,\mbox{d}t\mbox{d}s,\\
S_{2n}&=&\frac1n\sum_{j=1}^n\int_0^\infty\int_0^\infty st\mathbf{1}\{Y_j>s\lor t\}z_n(s)z_n(t)\exp(-a(s+t))\,\mbox{d}t\mbox{d}s\\
S_{3n}&=&\int_0^\infty\int_0^\infty (s+t)\Psi_{n,1}(s\lor t)z_n(s)z_n(t)\exp(-a(s+t))\,\mbox{d}t\mbox{d}s,\\
S_{4n}&=&\int_0^\infty\left(\Psi_{n,2}(t)-t\Psi_{n,1}(t)\right)z_n(t)\exp(-at)\mbox{d}t\\
S_{5n}&=&\frac1n\sum_{j=1}^n\int_0^\infty \mathbf{1}\{Y_j>t\}z_n(t)\exp(-at)\,\mbox{d}t,\\
S_{6n}&=&\int_0^\infty z_n^2(t)\exp(-at)\,\mbox{d}t,
\end{eqnarray*}
we have 
\begin{equation*}
\widehat{\sigma}_{n,a}^2=4\left[S_{1n}+S_{2n}-S_{3n}-2S_{4n}S_{5n}+\widehat{\tau}_n^2S_{5n}^2-S_{6n}^2\right].
\end{equation*}
Using the functions $\upsilon_{\ell,a}(x,y)$, $x,y>0$, $\ell=1,2,3$, from Appendix \ref{app:form}, we have
\begin{eqnarray*}
S_{1n}&=&\frac1{n^3}\sum_{j,k,l=1}^n(Y_j-1)^2\upsilon_{1,a}(Y_j,Y_k)\upsilon_{1,a}(Y_j,Y_l),\\
S_{2n}&=&\frac1{n^3}\sum_{j,k,l=1}^n\upsilon_{2,a}(Y_j,Y_k)\upsilon_{2,a}(Y_j,Y_l),\\
S_{3n}&=&\frac2{n^3}\sum_{j,k,l=1}^n(Y_j-1)\upsilon_{2,a}(Y_j,Y_k)\upsilon_{1,a}(Y_j,Y_l),\\
S_{4n}&=&\frac1{n^2}\sum_{j,k=1}^n\left[(Y_j-1)^2\upsilon_{1,a}(Y_j,Y_k)-(Y_j-1)\upsilon_{2,a}(Y_j,Y_k)\right],\\
S_{5n}&=&\frac1{n^2}\sum_{j,k=1}^n\upsilon_{1,a}(Y_j,Y_k),\quad\mbox{and}\quad S_{6n}=\frac1{n^2}\sum_{j,k=1}^n\upsilon_{3,a}(Y_j,Y_k).
\end{eqnarray*}
Since $\widehat{\sigma}^2_{n,a}$ is a consistent sequence of estimators of $\sigma^2_a$ for each fixed tuning parameter $a>0$, we have (in the spirit of Corollary 1 in \cite{BEH:2017}) under the stated assumptions at the beginning of this section
\begin{equation}\label{eq:asystnorm}
    \frac{\sqrt{n}}{\widehat{\sigma}_{n,a}}\left(\frac{T_{n,a}}{n}-\Delta_a\right)\cd \mbox{N}(0,1)\quad \mbox{as}\quad n\to\infty.
\end{equation}
This result has immediate consequences, see Section 3 of \cite{BEH:2017}, which are detailed in the following subsections.

\subsection{A confidence interval for $\Delta_a$}\label{subsec:conf}
For $\alpha\in(0,1)$ let $u_\alpha=\Phi^{-1}(1-\alpha/2)$ be the $(1-\alpha/2)$-quantile of the standard normal law. Putting
\begin{equation*}
    I_{n,a}=\left[\frac{T_{n,a}}{n}-\frac{u_\alpha\widehat{\sigma}_{n,a}}{\sqrt{n}},\frac{T_{n,a}}{n}+\frac{u_\alpha\widehat{\sigma}_{n,a}}{\sqrt{n}}\right],\quad a\ge 0,
\end{equation*}
it follows from \eqref{eq:asystnorm} that 
\begin{equation*}
    \lim_{n\to \infty}\mathbb{P}_F(I_{n,a}\ni \Delta_a)=1-\alpha.
\end{equation*}
Hence $I_{n,a}$ is an asymptotic confidence interval at confidence level $1-\alpha$ for $\Delta_a$. In the following we revisit the gamma distributions of Example \ref{ex:alt}. Specific values of $\Delta_a$, $\sigma_a^2$ and the corresponding estimators for $\Gamma(\beta,\beta)$ and $\beta\in\{2,3,4,5,10\}$ are found in Table \ref{tab:explval}. The discrepancy between estimators and true values are in line with the results of Table 5 in \cite{BEH:2017}, which indicates a rather slow convergence to the limiting distribution under alternatives in \eqref{eq:asystnorm}. In Table \ref{tab:coverage} we present the empirical coverage probabilities of $I_{n,a}$ for $\Delta_a$ in the same setting as in Table \ref{tab:explval}. Critical values have been obtained by the Pearson-system approximation as presented in Table \ref{tab:cv}. Interestingly the confidence interval seems to be conservative, since the estimated probability of coverage of the true value $\Delta_a$ is disproportionately high in most cases. This is in contrast to the findings in Table 6 in \cite{BEH:2017}, where lower coverage rates than indicated by the nominal level were reported.

\begin{table}[t]
\centering
\begin{tabular}{lrrrrrr}
    $a/\beta$  & 2 & 3 & 4 & 5 & 10 \\ \hline
    $\Delta_0$ & 0.0312  & 0.0648  &  0.0903 & 0.1100  & 0.1661 \\
    $T_{n,0}/n$ &  0.0252  & 0.0672  & 0.0944  & 0.1179  & 0.1778 \\
    $\sigma_0^2$ & 0.0178  & 0.0300 & 0.0357  & 0.0382  & 0.0362\\ 
    $\widehat{\sigma}_{n,0}^2$ & 0.0134  & 0.0313  & 0.0361  & 0.0377  & 0.0331\\ \hline
    $\Delta_1$ &  0.0160 & 0.0337  & 0.0472  &  0.0575 & 0.0863 \\
    $T_{n,1}/n$ &  0.0135 &  0.0349 & 0.0493  &  0.0619 &  0.0923\\
    $\sigma_1^2$ & 0.0039  &  0.0065 & 0.0076  & 0.0080  & 0.0070\\ 
    $\widehat{\sigma}_{n,1}^2$ & 0.0031  &  0.0069 &  0.0076 & 0.0080  & 0.0062 \\ \hline
    $\Delta_2$ &  0.0092 &  0.0193 & 0.0268  & 0.0324  & 0.0475 \\
    $T_{n,2}/n$ &  0.0081 &  0.0199 & 0.0280  &  0.0348 & 0.0506 \\
    $\sigma_2^2$ & 0.0012  & 0.0018  & 0.0020  & 0.0020  & 0.0015\\ 
    $\widehat{\sigma}_{n,2}^2$ & 0.0010  & 0.0019  & 0.0019  & 0.0020  & 0.0013 \\ \hline
    $\Delta_3$ & 0.0058  &  0.0119 & 0.0162  &  0.0194 & 0.0276 \\
    $T_{n,3}/n$ & 0.0053  & 0.0122  & 0.0170  &  0.0208 & 0.0292 \\
    $\sigma_3^2$ & 0.0004  &  0.0006 &  0.0006 & 0.0006  & 0.0003\\ 
    $\widehat{\sigma}_{n,3}^2$ & 0.0004  & 0.0006  & 0.0006  &  0.0006 & 0.0003 \\ \hline
    $\Delta_4$ & 0.0040  & 0.0078  & 0.0104  &  0.0123 & 0.0168 \\
    $T_{n,4}/n$ & 0.0036  &  0.0079 & 0.0109  &  0.0130 & 0.0177 \\
    $\sigma_4^2$ & 0.0002  &  0.0002 &  0.0002 &  0.0002 & 0.0001\\ 
    $\widehat{\sigma}_{n,4}^2$ & 0.0002  & 0.0002  & 0.0002  & 0.0002  & 0.0001
    \end{tabular}
\caption{Values of $\Delta_a$ and $\sigma_a^2$ for different $\Gamma(\beta,\beta)$ distributions and simulated observations of $T_{n,a}/n$ and $\widehat{\sigma}^2_{n,a}$ for different tuning parameters $a$ ($n=1000$).}\label{tab:explval}  
\end{table}

\begin{table}[b]
\centering
\begin{tabular}{llrrrrrrrrrrrrrrrrrr}
& $\beta$ &\multicolumn{3}{c}{2} & \multicolumn{3}{c}{3} & \multicolumn{3}{c}{4} & \multicolumn{3}{c}{5} & \multicolumn{3}{c}{10} \\[1mm]
$a$ & $n$ & 20 & 50 & 100 & 20 & 50 & 100 & 20 & 50 & 100 & 20 & 50 & 100 & 20 & 50 & 100 \\\hline 
0 & &0.95 & 0.95 & 0.94 & 0.96 & 0.96 & 0.97  &  0.96 & 0.98 & 0.98  &  0.97 & 0.98 & 0.99  &  0.98 & 0.99 & 0.99 \\
1 & & 0.90 & 0.92 & 0.93 & 0.92 & 0.95 & 0.95  &  0.93 & 0.96 & 0.96  &   0.94 & 0.96 & 0.97 &  0.96 & 0.98 & 0.99 \\
2 & & 0.87 & 0.90 & 0.91 & 0.89 & 0.93 & 0.94  & 0.90 & 0.94 & 0.95  &  0.91 & 0.94 & 0.96  & 0.93 & 0.96 & 0.97\\
3 & & 0.85 & 0.89 & 0.91 & 0.85 & 0.91 & 0.92 &  0.87 & 0.92 & 0.93  &   0.87 & 0.92 & 0.94  & 0.90 & 0.94 & 0.95  \\
4 &  & 0.82 & 0.88 & 0.90  &  0.83 & 0.90 & 0.92  &  0.84 & 0.90 & 0.92  &  0.85 & 0.90 & 0.93  &  0.87 & 0.92 & 0.94 
\end{tabular}
\caption{Empirical relative frequencies of coverage of $I_{n,a}$ for $\Delta_a$ at a nominal level of 0.9 for different $\Gamma(\beta,\beta)$ distributions (10000 replications).}\label{tab:coverage}  
\end{table}


\subsection{Neighbourhood-of-model validation}
A clear drawback in the field of goodness-of-fit testing is that if a level-$\alpha$-test does not lead to the rejection of the hypothesis $H_0$, the conclusion that $H_0$ is 'confirmed' is generally wrong. To overcome this problem, the results of \eqref{eq:altasyvert} lead to a so-called 'neighbourhood-of-model validation', see Subsection 3.3 in \cite{BEH:2017}. In this spirit one could see $\Delta_a$ as some sort of distance to the null hypothesis. If we argue to 'tolerate' a given value $\widetilde\Delta$, we can consider the testing problem
\begin{equation*}
    H_{\widetilde\Delta}:\Delta_a(F)\ge\widetilde\Delta\quad\mbox{versus}\quad K_{\widetilde\Delta}:\Delta_a(F)<\widetilde\Delta.
\end{equation*}
From \eqref{eq:altasyvert} we obtain an asymptotic level-$\alpha$-test by rejecting $H_{\widetilde\Delta}$ whenever
\begin{equation*}
    T_{n,a}\le n\widetilde\Delta-\sqrt{n}\widehat{\sigma}_{n,a}\Phi^{-1}(1-\alpha).
\end{equation*}
To prove this statement follow the reasoning in \cite{BEH:2017}, subsection 3.3.

\section{Asymptotic Bahadur efficiencies}\label{sec:ABE} 
In this section we only consider the case that $a>0$. We start the investigation by giving an alternative representation of $\Delta_a$ in Theorem \ref{thm:cons}. 
\begin{lemma}\label{lem:stoch_lim}
Suppose that $\mu=\mathbb{E}[X]$ and $\mathbb{E}[|X|^3]<\infty$. If $a>0$, we have
\begin{equation*}
\frac{T_{n,a}}n\cp\Delta_a=\frac2{a^3}-\frac1{a^3}\mathbb{E}\left[\left((1-|Y_1-Y_2|)(a^2+a)+a+2\right)\exp(-a Y_1\wedge Y_2)\right],
\end{equation*}
and if $a=0$, we have
\begin{equation*}
\frac{T_{n,0}}n\cp\Delta_0=\mathbb{E}\left[(Y_1\wedge Y_2)(Y_1-1)(Y_2-1)+(Y_1\wedge Y_2)^2\left(1-\frac12(Y_1+Y_2)\right)+\frac13(Y_1\wedge Y_2)^3\right],
\end{equation*}
where $Y_j=X_j/\mu$, $j=1,2$.
\end{lemma}
\begin{proof}
By Theorem \ref{thm:cons}, applied to the iid. random variables $Y_1,\ldots,Y_n$ with $\mathbb{E}[Y_1]=1$, we have 
\begin{equation*}
    \frac{T_n}n\cp \Delta_a=\int_0^\infty z^2(y)\exp(-ay)\,\mbox{d}y,
\end{equation*}
where $z(y)=\mathbb{E}[(Y_1-y-1)\mathbf{1}\{Y_1>y\}]$, $y>0$. Since $Y_1$ and $Y_2$ are independent, it follows that
\begin{equation*}
    z^2(y)=\mathbb{E}[(Y_1-y-1)(Y_2-y-1)\mathbf{1}\{Y_1\wedge Y_2>y\}].
\end{equation*}
Fubini's Theorem yields
\begin{eqnarray*}
    \Delta_a&=&\int_0^\infty\mathbb{E}[(Y_1-y-1)(Y_2-y-1)\mathbf{1}\{Y_1\wedge Y_2>y\}]\exp(-ay)\,\mbox{d}y\\
    &=&\mathbb{E}\left[\int_0^{Y_1\wedge Y_2}(Y_1-y-1)(Y_2-y-1)\exp(-ay)\,\mbox{d}y\right].
\end{eqnarray*}
Straightforward integration and some algebra concludes the proof in both cases.
\end{proof}
\begin{remark}
An alternative way of proving Lemma \ref{lem:stoch_lim} is to start by the representation \eqref{eq:expl} and to use symmetry arguments, the law of large numbers and Lebesgue's dominated convergence theorem, i.e. to adapt the lines of proof of Theorem 3.1 in \cite{EH:2021}.
\end{remark}
A useful tool for a theoretical comparison of the performance of two tests is the asymptotic relative Bahadur efficiency. This concept has been used throughout the literature on exponentiality tests, see \cite{CMO:2019,JMNOV:2015,M:2016,MO:2016,VN:2015}, and for more details on the theory we refer to \cite{B:1971,N:1995}. For a brief introduction to the concept, see Section 5 of \cite{JMO:2020}. In this spirit and using the same notations, we calculate the local approximate Bahadur slope of $T_{n,a}$. Assume that $0\in\Theta\subset\R$, where $\Theta$ is an open parameter space, and  $\mathcal{G}=\{G(x;\vartheta):\vartheta\in\Theta\}$ is a family of distribution functions with density $g(x;\vartheta)$, such that $\vartheta=0$ corresponds to the standard exponential density $g(x;0)=\exp(-x)$, $x>0$, and for each $\vartheta>0$ the density $g(x;\vartheta)$ is not a density corresponding to an exponential distribution in $\mathcal{E}$.  Moreover, we  assume that the regularity assumptions WD in \cite{NP:2004} are satisfied.
In the following assume that $X_1,X_2,\ldots$ are independent identical copies of $X$ following the distribution with density $g(\cdot;\vartheta)$ with existing expectation $\mathbb{E}[X]=\mu(\vartheta)=\int_0^\infty x g(x;\vartheta) \mbox{d}x$. Then Lemma \ref{lem:stoch_lim} yields
\begin{equation*}
\frac{T_{n,a}}n\stackrel{\mathbb{P}_\vartheta}{\longrightarrow} b_{T_a}(\vartheta),
\end{equation*}
where $\stackrel{\mathbb{P}_\vartheta}{\longrightarrow}$ denotes convergence in probability under the true parameter $\vartheta$, and \begin{equation*}
b_{T_a}(\vartheta)=\int_0^\infty\int_0^\infty h_a(x,y;\vartheta)g(x;\vartheta)g(y;\vartheta)\mbox{d}x\mbox{d}y
\end{equation*}
with
\begin{equation*}
a^3 h_a(x,y;\vartheta)=2-\left(\left(1-\frac{|x-y|}{\mu(\vartheta)}\right)(a^2+a)+a+2\right)\exp\left(-\frac{a}{\mu(\vartheta)}x\wedge y\right),\quad x,y>0,\;a>0,
\end{equation*}
and
\begin{equation*}
h_0(x,y;\vartheta)=\frac{x\wedge y}{\mu(\vartheta)}\left(\frac{x}{\mu(\vartheta)}-1\right)\left(\frac{y}{\mu(\vartheta)}-1\right)+\left(\frac{x\wedge y}{\mu(\vartheta)}\right)^2\left(1-\frac{x+y}{2\mu(\vartheta)}\right)+\frac1{3}\left(\frac{x\wedge y}{\mu(\vartheta)}\right)^3,\quad x,y>0.
\end{equation*}
Note that $\mu(0)=1$, $b_{T_a}(0)=0$, and after some algebra we have $b_{T_a}'(0)=0$. Here and in the following, all derivatives are calculated w.r.t. $\vartheta$. Writing $\mu_1=\mu'(0)$, the same reasoning as in Appendix B of \cite{CMO:2021} gives
\begin{eqnarray*}
    b_{T_a}''(0)&=& 2\int_0^\infty\int_0^\infty h_a(x,y;0)g'(x;0)g'(y;0)\mbox{d}x\mbox{d}y+4\mu_1\int_0^\infty\int_0^\infty h_a'(x,y;0)g(x;0)g'(y;0)\mbox{d}x\mbox{d}y\\
    && +\mu_1^2\int_0^\infty\int_0^\infty h_a''(x,y;0)g(x;0)g(y;0)\mbox{d}x\mbox{d}y.
\end{eqnarray*}
Expanding $b_{T_a}(\vartheta)$ into a Taylor series around $\vartheta_0=0$, we obtain
\begin{equation*}
    b_{T_a}(\vartheta)=\frac{b_{T_a}''(0)}{2}\vartheta^2+O(\vartheta^3),\quad\mbox{as}\;\vartheta\rightarrow 0.
\end{equation*}
From Section \ref{sec:Limit dist} we know that the limiting distribution of $T_n$ is $\|Z\|^2_{\mathbb{H}_a}=\sum_{j=1}^\infty \lambda_j(a) N_j^2,$ where $N_1,N_2,\ldots$ are iid. standard normal, and  $(\lambda_j(a))_{j\in\N}$ is the decreasing sequence of positive eigenvalues of the integral operator $\mathcal{K}_a$. Using the result in \cite{Z:1961},  the logarithmic tail behaviour of the limiting distribution of $\widetilde{T}_{n,a}=\sqrt{T_{n,a}}$ is
\begin{equation*}
    \log\left(1-F_{\widetilde{T}_a}(x)\right)=-\frac{x^2}{2\lambda_1(a)}+O(x^2),\quad x\rightarrow\infty.
\end{equation*}
Since the limit in probability of $\widetilde{T}_{n,a}/\sqrt{n}$ is $\sqrt{b_{T_a}(\vartheta)}$, the approximate local Bahadur slope is given by
\begin{equation*}
    c^*_{T_a}(\vartheta)=(\lambda_1(a))^{-1}b_{T_a}''(0)\vartheta^2+o(\vartheta^2),\quad \mbox{as}\,\vartheta\rightarrow0.
\end{equation*}
We compare the approximate Bahadur slope to the double Kullback-Leibler distance, also called Kullback-Leibler information numbers see \cite{NT:1996},
\begin{equation*}
    KL(g)=\int_0^\infty\frac{(g'(x;0))^2}{g(x;0)}\mbox{d}x-\left(\int_0^\infty G'(x;0)\,\mbox{d}x\right)^2,
\end{equation*}
where $G(x;\vartheta)=\int_0^xg(t,\vartheta)\mbox{d}t$ is the cumulative distribution function of $g(\cdot;\vartheta)$. It is well known that the Kullback-Leibler information numbers are an upper bound for Bahadur efficiencies, see \cite{B:1971,R:1970}. Hence we compute the approximate Bahadur efficiencies, given by
\begin{equation*}
    \mbox{eff}(g)=\frac{b_{T_a}''(0)}{2\lambda_1(a) KL(g)}.
\end{equation*}
These are equivalent to the comparison of the local approximate Bahadur slopes of $T_{n,a}$ and the likelihood ratio test as in \cite{CMO:2021}. Note that the largest eigenvalues $\lambda_1(a)$ are given in \eqref{eq:eig} and for easy reference tabulated in the first row of Table \ref{tab:eig}. To simplify the comparison to many well known competing procedures treated in \cite{CMO:2021}, we consider the following examples of distributions, all being a member of the class $\mathcal{G}$ and being standard references for the computation of Bahadur efficiencies of exponentiality tests, see Section 5 of \cite{N:1996}:
\begin{enumerate}
    \item the Weibull distribution with density
    \begin{equation*}
        g(x;\vartheta)=(1+\vartheta)x^\vartheta\exp\left(-x^{1+\vartheta}\right),\quad x\ge0,
    \end{equation*}
    and $KL(g)=1 - 2\gamma + \pi^2/6 + \gamma^2 - (1 - \gamma)^2$, where $\gamma=0.5772156649\ldots$ is the Euler-Mascheroni constant,
    \item the gamma distribution with density
    \begin{equation*}
        g(x;\vartheta)=x^\vartheta\exp\left(-x\right)/\Gamma(\vartheta+1),\quad x\ge0,
    \end{equation*}
    where $\Gamma(\cdot)$ denotes the gamma function, and $KL(g)=\pi^2/6 - 1$,
    
    \item a linear failure rate (LFR) distribution with density
    \begin{equation*}
        g(x;\vartheta)=(1+\vartheta x)\exp\left(-x-\vartheta x^2/2\right),\quad x\ge0,
    \end{equation*}
    and $KL(g)=1$,
    \item a mixture of exponential distributions with negative weights (EMNW$(\beta)$) with density
    \begin{equation*}
        g(x;\vartheta)=(1+\vartheta)\exp(-x)-\vartheta\beta\exp(-\beta x),\quad x\ge0,
    \end{equation*}
    and $KL(g)=16/45$ for $\beta=3$,
    \item and a Makeham distribution with density
    \begin{equation*}
        g(x;\vartheta)=(1 + \vartheta(1 - \exp(-x)))\exp\left(-x - \vartheta(x - 1 + \exp(-x))\right),\quad x\ge0,
    \end{equation*}
    and $KL(g)=1/12$.
\end{enumerate}

\begin{table}[t]
\centering
\begin{tabular}{lrrrrrr}
 Alt.$/a$& 0 & 1 & 2 & 3 & 4 & 5 \\ 
  \hline
  Weibull & 0.722 & 0.834 & 0.865 & 0.868 & 0.859 & 0.843 \\ 
  Gamma & 0.517 & 0.672 & 0.754 & 0.801 & 0.829 & 0.844 \\ 
  LFR & 0.917 & 0.731 & 0.592 & 0.495 & 0.424 & 0.371 \\ 
  EMNW$(3)$ & 0.765 & 0.940 & 0.987 & 0.982 & 0.954 & 0.917 \\ 
  Makeham & 0.918 & 0.987 & 0.948 & 0.884 & 0.818 & 0.757 
 \end{tabular}
\caption{Approximate Bahadur efficiencies $\mbox{eff}(g)$ of $T_{n,a}$ for different tuning parameters $a$.}\label{tab:eff}
\end{table}
The results are reported in Table \ref{tab:eff}. Interestingly, there is a clear dependence of the efficiency of the tests on the tuning parameter $a$ under all considered alternatives . Sometimes the highest efficiency is attained for the largest considered tuning parameter as in the gamma case, but for the LFR alternative for the lowest value of $a$. The Makeham alternative suggests to take 1 as best value for the tuning parameter. This behaviour is consistent with the empirical power study results in Tables 1 to 5 in \cite{BH:2008}. A comparison to the Bahadur efficiencies stated in Tables 3 and 4 in \cite{CMO:2021} for other tests of exponentiality shows that the considered procedures are competitive especially for the LFR and EMNW(3) alternatives. The results in Table \ref{tab:eff} confirm the suggestion of \cite{BH:2008} that a tuning parameter $a$ between 1 and 2 is a good choice, since it shows a robust approximate Bahadur efficiency over all considered alternatives.

\section{Conclusions and Outlook}\label{sec:conc}
We revisited the family of tests of exponentiality of \cite{BH:2008} and provided new insight into the asymptotic behaviour of the tests under fixed alternatives as well as local Bahadur efficiencies. These results facilitate the comparison of the performances to other well known tests of exponentiality. As a result we visualised the dependence of the power of the tests on the choice of the tuning parameter $a$. This effect might be controlled by implementing a data dependent choice of the tuning parameter due to \cite{T:2019}. We leave this investigation open for further research. 

We finish the article by pointing out other related open questions. 
As stated in Remark \ref{rem:Intro}, the family of tests based on the integrated distribution function presented in \cite{K:2001} is very close to the test statistic $T_{n,a}$. For this family of tests corresponding theoretical results as in Section \ref{sec:Limit_dist_fixed_alt} and \ref{sec:ABE} are missing. There is little hope to solve the eigenvalue problem in this case, but results regarding Bahadur efficiency may be obtained by approximating the largest eigenvalue numerically. Another open question is due to the conservative behaviour of the confidence interval $I_{n,a}$ in Subsection  \ref{subsec:conf}, which suggests that an improvement in view of the length of the confidence interval might be found. We leave this investigation open for further research.

In \cite{SAS:2019} the authors propose a test of exponentiality based on a characterisation of the exponential law by a conditional second moment equation involving the hazard rate. Theoretical insight into this family of tests are hitherto missing, so it would be interesting to see corresponding results, since the tests based on the mean residual life function are tests based on a characterisation by a first conditional moment equation. 

\section*{Acknowledgement} 
The author thanks Bernhard Klar and Bojana Milo{\v{s}}evi{\'{c}} for fruitful discussions, and is grateful to Norbert Henze for numerous suggestions that all led to an improvement of the paper. 

\bibliographystyle{apalike2}
\bibliography{references}

\begin{appendix}
\section{formulas for the functions in the estimator of the limiting variance}\label{app:form}
In this section we provide explicit formulas needed in Section \ref{sec:Limit_dist_fixed_alt}. We write for $x,y>0$ and $a>-1$
\begin{eqnarray*}
\upsilon_{1,a}(x,y)&=&\int_{0}^\infty(y-t-1)\mathbf{1}\{x\land y>t\}\exp(-at)\mbox{d}t,\\
\upsilon_{2,a}(x,y)&=&\int_{0}^\infty t(y-t-1)\mathbf{1}\{x\land y>t\}\exp(-at)\mbox{d}t,\quad \mbox{and}\\
\upsilon_{3,a}(x,y)&=&\int_0^\infty (y-t-1)(x-t-1)\mathbf{1}\{x\land y>t\}\exp(-at)\mbox{d}t.
\end{eqnarray*}
Straightforward integration and some calculations show for $a\not=0$
\begin{equation*}
    \upsilon_{1,a}(x,y)=a^{-2}\left[\left( a\left( x\land y \right)+1+ \left( 1-y \right) a
 \right) \exp(-a(x\land y))-1+ \left( y-1 \right)a\right]
\end{equation*}
and $\upsilon_{1,0}(x,y)=-(x\land y)  \left( -2y+(x\land y) +2
 \right)/2 $, $x,y>0$, as well as 
\begin{equation*}
    \upsilon_{2,a}(x,y)={a}^{-3}\left[\left(  \left( x\land y  \right) ^{2}{a}^{2}+
 \left(  \left( 1-y \right) {a}^{2}+2\,a \right) (x\land y) +2+ \left( 1-y \right) a \right)  \exp\left(-a(x\land y)\right)-2+ \left( y-1 \right) a\right]
\end{equation*}
and $\upsilon_{2,0}(x,y)=-\left( x\land y  \right) ^{2} \left( 2(x\land y)-3y+3 \right)/6 $, $x,y>0$. The third function integrates to
\begin{eqnarray*}
\upsilon_{3,a}(x,y)&=&a^{-3}\bigg[\big( -{a}^{2} \left( x\land y  \right) ^{2}+
a \left( -2+ \left( x+y-2 \right) a \right)\left( x \land y \right) -2
- \left( y-1 \right)  \left( x-1 \right) {a}^{2}\\
 &&+ \left( x+y-2
 \right) a \big)\exp\left(-a\left( x\land y \right) \right)+2+ \left( y- 1 \right)  \left( x-1 \right) {a}^{2}+ \left( -x-y+2 \right) a\bigg],
\end{eqnarray*}
where $\upsilon_{3,0}(x,y)=(x\land y)  \left(\left( x\land y  \right) ^{2}/3+ \left( -x/2-y/2+1 \right) (x\land y) +\left( y-1 \right)  \left( x-1 \right)  \right)$.

\section{\texttt{R} source code} \label{app:SC}
In the following we provide the source code written for the statistical computing language \texttt{R}, see \cite{R:2021}. An efficient implementation of the test statistic in \eqref{eq:expl} is given by the following code.

\begin{verbatim}
T.n.a<-function(data,a)
{
  n=length(data)
  data=data/mean(data)
  datam=matrix(data,n,n)
  pjl=datam+t(datam)
  mjl=(data-1)%*%(t(data)-1)
  minjl=pmin(datam,t(datam))
  if (a==0) {SUM=minjl^3/3-(pjl-2)*minjl^2/2+mjl*minjl} else {
  SUM=(-(datam-minjl-1)*(t(datam)-minjl-1)*exp(-a*minjl)+mjl)/a+((pjl
        -2*minjl-2)*exp(-a*minjl)-(pjl-2))/a^2+2*(1-exp(-a*minjl))/a^3}
  return(sum(SUM)/n)
}
\end{verbatim}

The following code can be used for approximation of the critical values of the test statistic by a Pearson system of distributions using the \texttt{R} package \texttt{PearsonDS}, see \cite{BK:2017}. 
\begin{verbatim}
#Cumulants from Remark 2.2
kappa_1<-function(a) {return(1/((a+1)*(a+2)))}
kappa_2<-function(a) {return(2/((a^2+3*a+2)*(2*a+3)*(a+2)))}
kappa_3<-function(a) {return(16/((a+1)*(a+2)^3*(2*a+3)*(3*a+4)))}
kappa_4<-function(a) {return(48*(11*a+16)/((a+1)*(a+2)^4*(2*a+3)^2*(3*a+4)*(4*a+5)))}

#The function provides the approximation of the 1-alpha quantile of the limiting
#distribution for a tuning parameter a
cv.T<-function(alpha,a)
{
  require(PearsonDS)
  kum=c(kappa_1(a),kappa_2(a),kappa_3(a),kappa_4(a))
  mom.a=c(kum[1:2],kum[3]*kum[2]^(-3/2),3+kum[4]*kum[2]^(-2))
  return(qpearson(1-alpha,moments=mom.a))
}
\end{verbatim}

\end{appendix}
\end{document}